\documentclass{amsart}

\usepackage{amsfonts,amsmath,amssymb,amsthm}
\usepackage{color}
\usepackage[hypertex,colorlinks,backref]{hyperref}
\usepackage{cite}

\newtheorem{theorem}{Theorem}
\theoremstyle{plain}

\newtheorem{corollary}{Corollary}

\newtheorem{proposition}{Proposition}
\newtheorem{remark}{Remark}

\numberwithin{equation}{section}

\begin{document}
\title{Ricci solitons on singly warped product manifolds and applications}
\author{Uday Chand De}
\address[U. C. De]{Department of Pure Mathematics, University of Calcutta,
35 Ballygaunge Circular Road, Kolkata 700019, West Bengala, India}
\email{uc$\_$de@yahoo.com}
\author{Carlo Alberto Mantica}
\address[C. Mantica]{I.I.S. Lagrange, Via L. Modignani 65, 20161 Milan,
Italy,}
\email{carloalberto.mantica@libero.it}
\author{Sameh Shenawy}
\address[S. Shenawy]{Basic Science Department, Modern Academy for
Engineering and Technology, Maadi, Egypt,}
\email{drssshenawy@eng.modern-academy.edu.eg, drshenawy@mail.com}
\author{B\"ulent \"Unal}
\address[B. \"{U}nal]{Department of Mathematics, Bilkent University,
Bilkent, 06800 Ankara, Turkey}
\email{bulentunal@mail.com}
\subjclass[2010]{53C25, and 53C40.}
\keywords{Warped product manifolds, Ricci solitons, concurrent vector fields}

\begin{abstract}
The purpose of this article is to study implications of a
Ricci soliton warped product manifold to its base and fiber
manifolds. First, it is proved that if a warped product
manifold is Ricci soliton then its factors are Ricci soliton.
Then we study Ricci soliton on warped product manifolds
admitting either a conformal vector field or a concurrent
vector field. Finally, we study Ricci soliton on some warped
product space-times.
\end{abstract}

\date{December 12, 2020}

\maketitle

\section{Introduction}

A Riemannian manifold $(M,g)$ is said to admit a Ricci soliton structure,
denoted by $\left( M,g,\zeta ,\lambda \right) $, if there exists a vector
field $\zeta \in \mathfrak{X}\left( M\right) $ and a scalar $\lambda $
satisfying
\begin{equation}
\frac{1}{2}\mathfrak{\mathcal{L}}_{\zeta }g+\mathrm{\mathrm{Ric}}=\lambda g
\label{RS}
\end{equation}%
where $\mathrm{\mathrm{Ric}}$ denotes the Ricci tensor of $M$ and $\mathcal{L%
}_{\zeta }$ denotes the Lie derivative in the direction of $\zeta $. A Ricci
soliton is said to be shrinking, steady or expanding if the scalar $\lambda $
is positive, zero or negative respectively $\ $If $\zeta =\mathrm{grad}u,$
for a smooth function $u$, the Ricci soliton $\left( M,g,\zeta ,\lambda
\right) \equiv \left( M,g,u,\lambda \right) $ is called a gradient Ricci
soliton and the function $u$ is called the potential function. Gradient
Ricci solitons are natural generalizations of Einstein manifolds \cite%
{Besse2008}. The study of Ricci solitons was first introduced by Hamilton as
fixed or stationary points of the Ricci flow in the space of the metrics on $%
M$ modulo diffeomorphisms and scaling. Since then, Ricci solitons have been
extensively studied for different reasons and in different spaces\cite%
{Brendle:2014, Cao:2010,
Chen:2017,Fernandez:2011,Grio2013,Manev20,Munteanu:2013,Peterson:2009a,Petersen:2010}%
. A large and growing body of research has continued to study Ricci soliton
after Pereleman used Ricci soliton to solve the Poincare conjecture posed in
1904.

Generally speaking, it is possible to categorize the research problems on
this topic under the perspective of warped product manifolds into two ways:

\textbf{1.} Under what conditions does the warped product become a Ricci
soliton?

\textbf{2.} What does a factor of a warped product Ricci soliton inherit?

There are many partial answers for these questions. For example, if a
gradient soliton splits $\left( M,g,f,\lambda \right) =\left( M_{1}\times
M_{2},g_{1}\oplus f^{2}g_{2},u,\lambda \right) $ as a Riemannian product,
then $u\left( x_{1},x_{2}\right) =u_{1}\left( x_{1}\right) +u_{2}\left(
x_{2}\right) $ also splits in such a way that each $\left(
M_{i},g_{i},u_{i},\lambda \right) $ is a soliton \cite{Petersen:2009b}. In
\cite[Sections 3 and 4]{Kim:2013}, the authors obtain a criteria that the
Riemannian manifold $M$ is Einstein or a gradient Ricci soliton using of the
second derivative of warping function $f$ in the warped and Lorentzian
warped product space of the form $%
\mathbb{R}
\times _{f}M$ with gradient Ricci solitons. In \cite{Feitosa:2017}, it is
shown that a non-shrinking gradient Ricci soliton warped product whose
warping function attains its extremes is a Riemannian product. Moreover,
existence conditions for a warped product gradient Ricci soliton are derived
on the warping function, a gradient vector field and on the fiber. The
authors of \cite[Section 3]{Lee:2016} derived conditions on the gradient
warped product Ricci soliton to have Einstein base manifold or to have
Einstein fiber manifold. They also considered warped product Ricci soliton
either with one dimensional Euclidean base or with one dimensional circle
base in \cite{Lee:2017}. Both necessary and sufficient conditions for
multiply warped product manifolds to be gradient Ricci solitons are obtained
in \cite{Karac:2018}. Special doubly warped gradient Ricci soliton with
harmonic Weyl tensor are considered in \cite{Kim:2016} where a doubly warped
product means a multiply warped product with two fibers. In \cite%
{Brozos:2012}, it is shown that locally conformally flat Lorentzian gradient
Ricci soliton is locally isometric to a Robertson--Walker warped product, if
the gradient of the potential function is non-null. In \cite{Sousa17},
the concept of gradient Ricci solitons on a semi-Riemannian warped product is
studied and it is proved that the potential function depends only on the base
factor of the underlying warped product and its fiber is an Einstein manifold.

It is clear that the work is restricted either to gradient warped Ricci
solitons or to special cases of warped Ricci solitons. In the current study,
we intend to fill this gap in the literature by providing a complete study
of warped product Ricci solitons. Moreover, we study warped product Ricci
solitons admitting conformal, concurrent or Killing vector fields. Finally,
we apply our results on generalized Robertson-Walker space-times and
standard static space-times.

\section{Preliminaries}

Let $\left( M_{i},g_{i},D^{i}\right) ,i=1,2$ be two $C^{\infty }$
pseudo-Riemannian manifolds equipped with metric tensors $g_{i}$ where $%
D^{i} $ is the Levi-Civita connection of the metric $g_{i}$ for $i=1,2.$ Let
$\pi _{1}:M_{1}\times M_{2}\rightarrow M_{1}$ and $\pi _{2}:M_{1}\times
M_{2}\rightarrow M_{2}$ be the natural projection maps of the Cartesian
product $M_{1}\times M_{2}$ onto $M_{1}$ and $M_{2}$ respectively. Also, let
$f:M_{1}\rightarrow \left( 0,\infty \right) $ be a positive real-valued
smooth function. The warped product manifold $M_{1}\times _{f}M_{2}$ is the
the product manifold $M_{1}\times M_{2}$ equipped with the metric tensor $%
g=g_{1}\oplus f^{2}g_{2}$ defined by%
\begin{equation*}
g=\pi _{1}^{\ast }\left( g_{1}\right) \oplus \left( f\circ \pi _{1}\right)
^{2}\pi _{2}^{\ast }\left( g_{2}\right)
\end{equation*}%
where $^{\ast }$ denotes the pull-back operator on tensors\cite%
{Bishop:1969,Oneill:1983,Shenawy:2015}. The function $f$ is called the
warping function of the warped product manifold $M_{1}\times _{f}M_{2}$. In
particular, if $f=1$, then $M_{1}\times _{1}M_{2}=M_{1}\times M_{2}$ is the
usual Cartesian product manifold. It is clear that the submanifold $%
M_{1}\times \{q\}$ is isometric to $M_{1}$ for every $q\in M_{2}$. Also, $%
\{p\}\times M_{2}$ is homothetic to $M_{2}$ for every $p\in M_{1}$.
Throughout this article, we use the same notation for a tensor field and for
its lift to the product manifold. Let $D$ be the Levi-Civita connection of
the metric tensor $g$. In \cite{Bishop:1969,Oneill:1983, Shenawy:2015}, the
Ricci curvature $\mathrm{\mathrm{Ric}}$ of the warped product manifold in
terms of the lift of Ricci curvatures, $\mathrm{\mathrm{Ric}}^{i}$, on $%
M_{i},$ for $i=1,2$ and $\mathrm{H}^{f}$, the Hessian tensor of $f$ on $%
M_{1} $ is described.

\section{Warped product Ricci solitons}

This section presents a study of warped product Ricci soliton. The following
result considers the first inheritance property.

\begin{theorem}
Let $(M,g,\zeta ,\lambda )$ be a Ricci soliton where $(M,g)=(M_{1}\times
_{f}M_{2},g_{1}\oplus f^{2}g_{2})$. Then,

\begin{enumerate}
\item $\left( M_{1},g_{1},\zeta _{1}-\eta ,\lambda \right) $ is a Ricci
soliton provided that $\frac{n_{2}}{f}D_{X_{1}}^{1}\nabla
^{1}f=D_{X_{1}}^{1}\eta _{1}$, and

\item $\left( M_{2},g_{2},f^{2}\zeta _{2},\lambda f^{2}+f^{\sharp }-f\zeta
_{1}\left( f\right) \right) $ is a Ricci soliton when $\lambda
f^{2}+f^{\sharp }-f\zeta _{1}\left( f\right) $ is constant.
\end{enumerate}
\end{theorem}

\begin{proof}
It is well-known that%
\begin{eqnarray}
\mathfrak{\mathcal{L}}_{\zeta }g &=&\mathcal{L}_{\zeta _{1}}^{1}g_{1}+f^{2}%
\mathcal{L}_{\zeta _{2}}^{2}g_{2}+2f\zeta _{1}\left( f\right) g_{2}, \\
\mathrm{\mathrm{Ric}} &=&\mathrm{Ric}^{1}-\frac{n_{2}}{f}\mathrm{H}^{f}+%
\mathrm{Ric}^{2}-f^{\sharp }g_{2},
\end{eqnarray}%
where $f^{\sharp }=f\Delta f+\left( n_{2}-1\right) \left\Vert \nabla
f\right\Vert _{1}^{2}$. Thus Equation (\ref{RS}) becomes%
\begin{eqnarray}
\lambda g_{1}+\lambda f^{2}g_{2} &=&\frac{1}{2}\mathcal{L}_{\zeta
_{1}}^{1}g_{1}+\frac{1}{2}f^{2}\mathcal{L}_{\zeta _{2}}^{2}g_{2}+f\zeta
_{1}\left( f\right) g_{2}+\mathrm{Ric}^{1}  \notag \\
&&-\frac{n_{2}}{f}\mathrm{H}^{f}+\mathrm{Ric}^{2}-f^{\sharp }g_{2}.
\label{RS1}
\end{eqnarray}%
It is noted that%
\begin{eqnarray}
\left( \frac{1}{2}\mathcal{L}_{\zeta _{1}}^{1}g_{1}-\frac{n_{2}}{f}\mathrm{H}%
^{f}\right) \left( X_{1},Y_{1}\right)  &=&\frac{1}{2}g_{1}\left(
D_{X_{1}}^{1}\zeta _{1},Y_{1}\right) +\frac{1}{2}g_{1}\left(
X_{1},D_{Y_{1}}^{1}\zeta _{1}\right)   \notag \\
&&-\frac{n_{2}}{f}g_{1}\left( D_{X_{1}}^{1}\nabla ^{1}f,Y_{1}\right) .
\label{02}
\end{eqnarray}%
However, $g_{1}\left( D_{X_{1}}^{1}\nabla ^{1}f,Y_{1}\right) =g_{1}\left(
D_{Y_{1}}^{1}\nabla ^{1}f,X_{1}\right) $. Thus,%
\begin{eqnarray*}
\frac{n_{2}}{f}g_{1}\left( D_{X_{1}}^{1}\nabla ^{1}f,Y_{1}\right)  &=&\frac{%
n_{2}}{2f}g_{1}\left( D_{X_{1}}^{1}\nabla ^{1}f,Y_{1}\right) +\frac{n_{2}}{2f%
}g_{1}\left( D_{X_{1}}^{1}\nabla ^{1}f,Y_{1}\right)  \\
&=&\frac{n_{2}}{2f}g_{1}\left( D_{X_{1}}^{1}\nabla ^{1}f,Y_{1}\right) +\frac{%
n_{2}}{2f}g_{1}\left( D_{Y_{1}}^{1}\nabla ^{1}f,X_{1}\right)
\end{eqnarray*}%
Now, assume that there is a vector field $\eta _{1}$ on the base manifold
such that $\frac{n_{2}}{f}D_{X_{1}}^{1}\nabla ^{1}f=D_{X_{1}}^{1}\eta _{1}$,
then Equation (\ref{02}) reads as
\begin{eqnarray*}
\left( \frac{1}{2}\mathcal{L}_{\zeta _{1}}^{1}g_{1}-\frac{n_{2}}{f}\mathrm{H}%
^{f}\right) \left( X_{1},Y_{1}\right)  &=&\frac{1}{2}g_{1}\left(
D_{X_{1}}^{1}\zeta _{1},Y_{1}\right) -\frac{n_{2}}{2f}g_{1}\left(
D_{X_{1}}^{1}\nabla ^{1}f,Y_{1}\right)  \\
&&+\frac{1}{2}g_{1}\left( X_{1},D_{Y_{1}}^{1}\zeta _{1}\right) -\frac{n_{2}}{%
2f}g_{1}\left( X_{1},D_{Y_{1}}^{1}\nabla ^{1}f\right)  \\
&=&\frac{1}{2}g_{1}\left( D_{X_{1}}^{1}\left( \zeta _{1}-\eta _{1}\right)
,Y_{1}\right) +\frac{1}{2}g_{1}\left( X_{1},D_{Y_{1}}^{1}\left( \zeta
_{1}-\eta _{1}\right) \right)  \\
&=&\frac{1}{2}\left( \mathcal{L}_{\zeta _{1}-\eta _{1}}^{1}g_{1}\right)
\left( X_{1},Y_{1}\right) .
\end{eqnarray*}%
Thus, Equation (\ref{RS1}) may be rewritten as%
\begin{equation*}
\lambda g_{1}+\left( \lambda f^{2}+f^{\sharp }-f\zeta _{1}\left( f\right)
\right) g_{2}=\frac{1}{2}\mathcal{L}_{\zeta _{1}-\eta _{1}}^{1}g_{1}+\mathrm{%
Ric}^{1}+\frac{1}{2}f^{2}\mathcal{L}_{\zeta _{2}}^{2}g_{2}+\mathrm{Ric}^{2}.
\end{equation*}%
Since $f$ is a function on the base factor, it implies that $f^{2}\mathcal{L}_{\zeta
_{2}}^{2}g_{2}=\mathcal{L}_{f^{2}\zeta _{2}}^{2}g_{2}$ and hence%
\begin{equation*}
\lambda g_{1}+\left( \lambda f^{2}+f^{\sharp }-f\zeta _{1}\left( f\right)
\right) g_{2}=\frac{1}{2}\mathcal{L}_{\zeta _{1}-\eta _{1}}^{1}g_{1}+\mathrm{%
Ric}^{1}+\frac{1}{2}\mathcal{L}_{f^{2}\zeta _{2}}^{2}g_{2}+\mathrm{Ric}^{2}.
\end{equation*}%
Now, when the arguments are restricted to the factor manifolds, one can obtain%
\begin{eqnarray*}
\lambda g_{1} &=&\frac{1}{2}\mathcal{L}_{\zeta _{1}-\eta _{1}}^{1}g_{1}+%
\mathrm{Ric}^{1}, \\
\left( \lambda f^{2}+f^{\sharp }-f\zeta _{1}\left( f\right) \right) g_{2} &=&%
\frac{1}{2}\mathcal{L}_{f^{2}\zeta _{2}}^{2}g_{2}+\mathrm{Ric}^{2},
\end{eqnarray*}%
which completes the proof.
\end{proof}

\begin{remark}
In the preceding theorem, the properties of the vector field $\eta _{1}$
depend on the geometry of the base manifold $M_{1}$ and the warping function $f$.
For that, we consider the following cases:

\begin{enumerate}
\item Assume that $M_{1}$ is compact and $\eta _{1}$ is a gradient vector
field. Then, there is a smooth function $\phi $ on $M_{1}$, such that $\eta
_{1}=\nabla ^{1}\phi $ and so $D_{Y_{1}}^{1}\nabla ^{1}\phi =\frac{1}{f}%
D_{Y_{1}}^{1}\nabla ^{1}f$. Hence, $\Delta _{1}\left( \phi -\log f\right)
=\left\vert \nabla ^{1}f\right\vert ^{2}/f^{2}$. The integration of both
sides over $M_{1}$ implies that $\nabla ^{1}f=0,$ i.e., $f$ is constant, say $%
f=1$, and consequently $M$ is a product manifold. The soliton equations on
the factor manifolds become%
\begin{eqnarray*}
\lambda g_{1} &=&\frac{1}{2}\mathcal{L}_{\zeta _{1}}^{1}g_{1}+\mathrm{Ric}%
^{1}, \\
\lambda g_{2} &=&\frac{1}{2}\mathcal{L}_{\zeta _{2}}^{2}g_{2}+\mathrm{Ric}%
^{2}.
\end{eqnarray*}

\item The warped product manifold $M=I\times _{e^{t}}M_{2},$ where $M_{1}=I$
is an open connected subinterval of $%
\mathbb{R}
$, admits a vector field $\eta =t\partial _{t}$ where $\frac{1}{f}%
D_{\partial _{t}}^{1}\nabla ^{1}f=D_{\partial _{t}}^{1}\eta _{1}$.

\item Assume that $M_{1}$ has a constant curvature $\kappa _{1}$. In \cite%
{Levy:1925}, it is proved that a parallel symmetric $\left( 0,2\right) $
tensor in a manifold of constant curvature is a constant multiple of the
metric tensor. Let us define the tensor $T$ as%
\begin{equation*}
T\left( X_{1},Y_{1}\right) =g_{1}\left( D_{Y_{1}}^{1}\eta _{1},X_{1}\right) =%
\frac{1}{f}g_{1}\left( D_{Y_{1}}^{1}\nabla ^{1}f,X_{1}\right) .
\end{equation*}%
It is clear that $T$ is symmetric and so the parallelism of $T$ is
sufficient for $D_{Y_{1}}^{1}\eta =cY_{1}$ for some constant $c$ and $%
\mathcal{L}_{\eta _{1}}^{1}g_{1}=2cg_{1}$. Thus, the soliton equation
reduces to%
\begin{eqnarray*}
\left( \lambda +c\right) g_{1} &=&\frac{1}{2}\mathcal{L}_{\zeta
_{1}}^{1}g_{1}+\mathrm{Ric}^{1}, \\
\left( \lambda +c\right) g_{1} &=&\frac{1}{2}\mathcal{L}_{\zeta
_{1}}^{1}g_{1}+\left( n_{1}-1\right) \kappa _{1}g_{1} \\
\left[ \lambda +c-\left( n_{1}-1\right) \kappa _{1}\right] g_{1} &=&\frac{1}{%
2}\mathcal{L}_{\zeta _{1}}^{1}g_{1},
\end{eqnarray*}%
that is, $\zeta _{1}$ is also a conformal vector field on $M_{1}$.
\end{enumerate}
\end{remark}

Let $\omega $ be a one-form metrically equivalent to $\eta _{1}$. It is
clear that $g_{1}\left( X_{1},D_{Y_{1}}^{1}\eta _{1}\right) =g_{1}\left(
Y_{1},D_{X_{1}}^{1}\eta _{1}\right) $ and consequently $\omega $ is closed.
In simply connected manifolds, every closed one-form is exact. A direct
consequence of the above theorem is the following corollary.

\begin{corollary}
Let $(M,g,\zeta ,\lambda )$ be a Ricci soliton where $(M,g)=(M_{1}\times
_{f}M_{2},g_{1}\oplus f^{2}g_{2})$ and $\zeta _{1}$ is homothetic on $M_{1},$
that is $\mathcal{L}_{\zeta _{1}}^{1}g_{1}=2ag_{1}$ for some constant $a$.
Then, $\left(M_{1},g_{1},-\eta _{1},\lambda -a\right) $ is a gradient Ricci soliton
whenever $M_{1}$ is simply connected.
\end{corollary}

It is clear that a Ricci soliton $\left( M,g,\zeta ,\lambda \right) $ is
Einstein with factor $\lambda -\rho $ if and only if $\zeta $ is conformal
with factor $2\rho $.

\begin{theorem}
\label{thm-KVF} Let $(M,g,\zeta ,\lambda )$ be a Ricci soliton where $%
(M,g)=(M_{1}\times _{f}M_{2},g_{1}\oplus f^{2}g_{2})$. Then $\left(
M,g\right) $ is Einstein if one of the following conditions holds

\begin{enumerate}
\item $\zeta _{i}$ is conformal on $M_{i}$ with factor $2\rho _{i}$ for any $%
i=1,2$ and $\rho _{1}=\rho _{2}+\zeta _{1}\left( \ln f\right) $

\item $\zeta =\zeta _{1}$ and $\zeta _{1}$ is a Killing vector field on $%
M_{1}$.

\item $\zeta =\zeta _{2}$ and $\zeta _{2}$ is a Killing vector field on $%
M_{2}$.

\item $\zeta _{i}$ is a Killing vector field on $M_{i}$ for $i=1,2$ and $%
\zeta _{1}\left( f\right) =0$.
\end{enumerate}
\end{theorem}

The following theorem considers the converse.

\begin{theorem}
\label{thm-Riccisol} Let $\left( M_{1},g_{1},\zeta _{1},\lambda \right) $ be
a Ricci soliton and $\left( M_{2},g_{2}\right) $ be an Einstein manifold
with factor $\mu $. Then $\left( M,g,\zeta ,\lambda \right) $ is a Ricci
soliton where $\left( M,g\right) $ is a warped product manifold of the form $%
\left( M_{1}\times _{f}M_{2},g_{1}\oplus f^{2}g_{2}\right) $ if

\begin{enumerate}
\item $\zeta _{2}$ is conformal with factor $2\rho ,$

\item $\mathrm{H}^{f}=0$ and,

\item $\left( \lambda -\rho \right) f^{2}=f\zeta _{1}\left( f\right) +\mu
+\left( n_{2}-1\right) c^{2}$ where $g_{1}\left( \nabla f,\nabla f\right)
=c^{2}$ for some $c\in \mathbb{R}.$
\end{enumerate}
\end{theorem}

A vector field $\zeta $ on a Riemannian manifold $M$ which satisfies%
\begin{equation*}
\nabla _{X}\zeta =X
\end{equation*}%
for any vector field $X\in \mathfrak{X}\left( M\right) $ is called a
concurrent vector field \cite{Chen:2015}. It is clear that a concurrent
vector field is a homothetic one with factor $\rho =2$ since $\left(
\mathcal{L}_{\zeta }g\right) \left( X,Y\right) =2g\left( X,Y\right) $.
Furthermore, a constant vector field is not concurrent. $\zeta $ is called
gradient if there is a smooth function $u$ defined on $M$ such that $\zeta
=\nabla u$. In this case $\left( \mathcal{L}_{\zeta }g\right) \left(
X,Y\right) =2\mathrm{H}^{u}\left( X,Y\right) $ where $\mathrm{H}^{u}$ is the
Hessian tensor of $u$ defined on $M$. Let $\zeta $ be a concurrent vector
field and let $u=\frac{1}{2}g\left( \zeta ,\zeta \right) $, then $\zeta
=\nabla u$. The converse is also true. A vector field $\zeta $ on a manifold
$M$ is concurrent with respect to a Riemannian metric $g$ if and only if $%
\zeta =\nabla u,$ and $\mathcal{L}_{\zeta }g=2g$ where $u=\frac{1}{2}g\left(
\zeta ,\zeta \right) $. Hence, a concurrent vector field is a gradient
vector field.

\begin{proposition}
\label{Thm-concurrent} Let $\zeta =\zeta _{1}+\zeta _{2}$ be a vector field
on $M=\left( M_{1}\times _{f}M_{2},g\right) $. $\zeta $ is concurrent on $M$
if and only if $\zeta _{1}$ is a concurrent vector field on $M_{1}$ and one
of the following conditions hold.

\begin{enumerate}
\item $\zeta _{2}$ is a concurrent vector field on $M_{2}$ and $f$ is
constant.

\item $\zeta _{2}=0$ and $\zeta _{1}\left( f\right) =f$.
\end{enumerate}
\end{proposition}

\begin{proof}
Suppose that $\zeta $ is a concurrent vector field on $M$. Then%
\begin{equation}
D_{\partial _{i}}\zeta =\partial _{i},i=1,2,...,n_{1},n_{1}+1,...,n_{1}+n_{2}
\label{CE1}
\end{equation}%
The first $n_{1}$ equations of (\ref{CE1}) imply that%
\begin{eqnarray*}
D_{\partial _{i}}\left( \zeta _{1}+\zeta _{2}\right) &=&\partial _{i} \\
D_{\partial _{i}}^{1}\zeta _{1}+\partial _{i}\left( \ln f\right) \zeta _{2}
&=&\partial _{i}
\end{eqnarray*}%
The tangential and normal parts of the last equation are%
\begin{eqnarray}
D_{\partial _{i}}^{1}\zeta _{1} &=&\partial _{i}  \label{CE2} \\
\partial _{i}\left( \ln f\right) \zeta _{2} &=&0  \label{CE3}
\end{eqnarray}%
Equation (\ref{CE2}) implies that $\zeta _{1}$ is concurrent vector field on
$M_{1}$. The second equation implies that%
\begin{equation*}
\partial _{i}\left( \ln f\right) =0\text{ \ \ \ or\ \ \ \ }\zeta _{2}=0
\end{equation*}%
Now, we have two cases:

\textbf{Case 1:} $\partial _{i}\left( \ln f\right) =0$ for any $%
i=1,2,...,n_{1}$: This equation implies that $f$ is constant. Now, we use
the second $n_{2}$ equations of (\ref{CE1})%
\begin{eqnarray*}
D_{\partial _{i}}\left( \zeta _{1}+\zeta _{2}\right) &=&\partial _{i}\text{
\ }i=n_{1}+1,n_{1}+2,...,n_{1}+n_{2} \\
\zeta _{1}\left( \ln f\right) \partial _{i}+D_{\partial _{i}}^{2}\zeta
_{2}-fg_{2}\left( \partial _{i},\zeta _{2}\right) \mathrm{grad}f &=&\partial
_{i}
\end{eqnarray*}%
Since $f$ is constant,%
\begin{equation*}
D_{\partial _{i}}^{2}\zeta _{2}=\partial _{i}
\end{equation*}%
i.e, $\zeta _{2}$ is a concurrent vector field on $M_{2}$. Thus the first
condition holds.

\textbf{Case 2:} $\zeta _{2}=0$: Now, we use the second $n_{2}$ equations of
(\ref{CE1})%
\begin{eqnarray*}
D_{\partial _{i}}\left( \zeta _{1}+\zeta _{2}\right) &=&\partial _{i} \quad
\text{for any} \quad i=n_{1}+1,n_{1}+2,...,n_{1}+n_{2} \\
\zeta _{1}\left( \ln f\right) \partial _{i} &=&\partial _{i}
\end{eqnarray*}%
This equation implies that $\zeta _{1}\left( \ln f\right) =1$ and hence $%
\zeta _{1}\left( f\right) =f$. This is the second condition.

Conversely, suppose that the first condition holds. Then for $%
i=1,2,...,n_{1} $ we have%
\begin{equation*}
D_{\partial _{i}}\zeta =D_{\partial _{i}}^{1}\zeta _{1}\zeta +\partial
_{i}\left( \ln f\right) \zeta _{2}=D_{\partial _{i}}^{1}\zeta _{1}=\partial
_{i}
\end{equation*}%
and for $i=n_{1}+1,n_{1}+2,...,n_{1}+n_{2}$ we have%
\begin{equation*}
D_{\partial _{i}}\zeta =\zeta _{1}\left( \ln f\right) \partial
_{i}+D_{\partial _{i}}^{2}\zeta _{2}-fg_{2}\left( \zeta _{2},\partial
_{i}\right) \mathrm{grad}f=D_{\partial _{i}}^{2}\zeta _{2}=\partial _{i}
\end{equation*}%
Therefore, $\zeta $ is a concurrent vector field. Now suppose that the
second condition holds. Then for $i=1,2,...,n_{1}$ we have%
\begin{equation*}
D_{\partial _{i}}\zeta =D_{\partial _{i}}\zeta _{1}=D_{\partial
_{i}}^{1}\zeta _{1}=\partial _{i}
\end{equation*}%
and for $i=n_{1}+1,n_{1}+2,...,n_{1}+n_{2}$ we have%
\begin{equation*}
D_{\partial _{i}}\zeta =\zeta _{1}\left( \ln f\right) \partial _{i}=\partial
_{i}
\end{equation*}%
Therefore $\zeta $ is concurrent in this case also and the proof is complete.
\end{proof}

\begin{remark}
The above result ensures that

\begin{enumerate}
\item There is no concurrent vector field on $M_{1}\times _{f}M_{2}$ of the
form $\zeta =\zeta _{2}$. Thus, there is no space-like concurrent vector
field on $I\times _{f}M$ and there is no time-like concurrent vector field
on standard static space-time $I_{f}\times M$.

\item The only time-like concurrent vector field on $I\times _{f}M$ is given
by $\zeta _{1}=\left( t+c\right) \partial _{t}$ where $f(t)=a\left(
t+c\right) $ and $a>0$.

\item The only concurrent vector field of the form $\zeta =\zeta _{1}$ on $%
M_{1}\times _{f}M_{2}$ exists if $\zeta _{1}$ is concurrent on $M_{1}$ and $%
\zeta _{1}\left( f\right) =f$.
\end{enumerate}
\end{remark}

Let $\bar{M}=I\times _{f}M$ be a generalized Robertson-Walker space-times
equipped with the metric $\bar{g}=-\mathrm{d}t^{2}\oplus f^{2}g$ where $%
\left( M,g\right) $ is a Riemannian manifold and $I$ is an open connected
interval with the usual flat metric $-\mathrm{d}t^{2}$. A vector field of
the form $\zeta =u\partial _{t}$ is a concurrent vector field on $\bar{M}$
if $u=\left( t+c\right) $ and $f=au$ where $a>0$ and $t+c>0$. However, the
vector field $\zeta =\coth t\partial _{t}$ on $I\times _{\cosh t}M$
satisfies that $\zeta \left( f\right) =f.$ But $\zeta $ is not concurrent on
$I\times _{\cosh t}M$ since $\zeta _{1}=\coth t\partial _{t}$ is not
concurrent on $I$. The rigorous of generalized Robertson-Walker space-times
will be given in Section 4.

\begin{theorem}
Let $\left( M,g,\zeta ,\lambda \right) $ be a Ricci soliton and $\zeta $ be
a concurrent vector field on $M$ where $\left( M,g\right) $ is a warped
product of the form $\left( M_{1}\times _{f}M_{2},g_{1}\oplus
f^{2}g_{2}\right) .$ If $\zeta _{2}\neq 0$, then $M,M_{1}$ and $M_{2}$ are
Ricci flat, gradient Ricci soliton with $\lambda =1$.
\end{theorem}

\begin{proof}
Let $\left( M_{1}\times _{f}M_{2},g,\zeta ,\lambda \right) $ be a Ricci
soliton and $\zeta $ be a concurrent vector field on $M_{1}\times _{f}M_{2}$%
. Then%
\begin{equation}
\mathrm{Ric}\left( X,Y\right) =\left( \lambda -1\right) g\left( X,Y\right)
\label{E6}
\end{equation}%
Suppose that $X=X_{2}$ and $Y=Y_{2}$, then%
\begin{equation*}
\mathrm{Ric}^{2}\left( X_{2},Y_{2}\right) =f^{\sharp }g_{2}\left(
X_{2},Y_{2}\right) +\left( \lambda -1\right) f^{2}g_{2}\left(
X_{2},Y_{2}\right)
\end{equation*}%
where $f^{\sharp }=f\Delta f+\left( n_{2}-1\right) \left\Vert \nabla
f\right\Vert _{1}^{2}$. Since $\zeta $ is concurrent and $\zeta _{2}\neq 0$,
$\zeta _{2}$ is concurrent and $f=c$ is constant. This implies that $%
f^{\sharp }=0$ and so%
\begin{equation}
\mathrm{\mathrm{Ric}}^{2}\left( X_{2},Y_{2}\right) =\left( \lambda -1\right)
c^{2}g_{2}\left( X_{2},Y_{2}\right)  \label{E4}
\end{equation}%
i.e. $M_{2}$ is Einstein with factor $\mu =\left( \lambda -1\right) f^{2}$.
This equation is true for any vector field in $\mathfrak{X}\left(
M_{2}\right) $ and so%
\begin{equation}
\mathrm{\mathrm{Ric}}^{2}\left( \zeta _{2},\zeta _{2}\right) =\left( \lambda
-1\right) c^{2}\left\Vert \zeta _{2}\right\Vert _{2}^{2}  \label{E5}
\end{equation}

Let $\{\zeta _{2},$ $e_{1},e_{2},...,e_{n_{2}-1}\}$ be an orthogonal basis
of $\mathfrak{X}\left( M_{2}\right) $, then the curvature tensor is given by%
\begin{eqnarray*}
R^{2}\left( \zeta _{2},e_{i},\zeta _{2},e_{i}\right) &=&g_{2}\left(
R^{2}\left( \zeta _{2},e_{i}\right) \zeta _{2},e_{i}\right) \\
&=&g_{2}\left( D_{\zeta _{2}}D_{e_{i}}\zeta _{2}-D_{e_{i}}D_{\zeta
_{2}}\zeta _{2}-D_{\left[ \zeta _{2},e_{i}\right] }\zeta _{2},e_{i}\right) \\
&=&g_{2}\left( D_{\zeta _{2}}e_{i}-D_{e_{i}}\zeta _{2}-\left[ \zeta
_{2},e_{i}\right] ,e_{i}\right) \\
&=&0
\end{eqnarray*}%
Thus $\mathrm{\mathrm{Ric}}^{2}\left( \zeta _{2},\zeta _{2}\right) =0$. By
substitution in equation (\ref{E5}) we get that $\lambda =1$ and so
equations \ref{E6} and \ref{E4} imply that $\mathrm{\mathrm{Ric}}\left(
X,Y\right) =\mathrm{\mathrm{Ric}}^{2}\left( X_{2},Y_{2}\right) =0$.

Now suppose that $X=X_{1}$ and $Y=Y_{1}$, then we get that%
\begin{eqnarray*}
\mathrm{\mathrm{Ric}}\left( X_{1},Y_{1}\right) &=&0 \\
\mathrm{\mathrm{Ric}}^{1}\left( X_{1},Y_{1}\right) &=&\frac{n_{2}}{f}\mathrm{%
H}^{f}\left( X_{1},Y_{1}\right) =0
\end{eqnarray*}%
It is easy to show that all of them are shrinking Ricci soliton with the
same factor $\lambda =1$. Moreover, $\zeta $ and $\zeta _{i}$ are gradient
vector fields with potential functions $u=\frac{1}{2}g\left( \zeta ,\zeta
\right) $ and $u_{i}=\frac{1}{2}g\left( \zeta _{i},\zeta _{i}\right) $ where
$i=1,2$ since%
\begin{eqnarray*}
g\left( X,\nabla u\right) &=&X\left( u\right) =g\left( D_{X}\zeta ,\zeta
\right) \\
&=&g\left( X,\zeta \right)
\end{eqnarray*}%
i.e. $\zeta =\nabla u$ and similarly $\zeta _{i}=\nabla _{i}u_{i}$.
\end{proof}

\begin{theorem}
\label{thm-imp} Let $\left( M,g,u,\lambda \right) $ be a gradient Ricci
soliton where $\left( M,g\right) $ is a warped product of the form $\left(
M_{1}\times _{f}M_{2},g_{1}\oplus f^{2}g_{2}\right) $. Then

\begin{enumerate}
\item $\left( M_{1},g_{1},\phi _{1},\lambda \right) $ is a gradient Ricci
soliton with $\phi _{1}=u_{1}-n_{2}\ln f$ and $u_{1}=u$ at some fixed point
of $M_{2}$.

\item $\left( M_{2},g_{2},\phi _{2},\lambda f^{2}\right) $ is a gradient
Ricci soliton with $\phi _{2}=u$ at some fixed point of $M_{1}$ if $f$ is
constant.
\end{enumerate}
\end{theorem}

\begin{proof}
Suppose that $\left( M_{1}\times _{f}M_{2},g,u,\lambda \right) $ is a
gradient Ricci soliton, then%
\begin{equation*}
\mathrm{H}^{u}\left( X,Y\right) +\mathrm{\mathrm{Ric}}\left( X,Y\right)
=\lambda g\left( X,Y\right)
\end{equation*}%
for any vector fields $X,Y\in \mathfrak{X}\left( M_{1}\times
_{f}M_{2}\right) $. Let $X=X_{1}$ and $Y=Y_{1}$, then%
\begin{eqnarray*}
\mathrm{H}^{u}\left( X_{1},Y_{1}\right) +\mathrm{\mathrm{Ric}}\left(
X_{1},Y_{1}\right) &=&\lambda g\left( X_{1},Y_{1}\right) \\
\mathrm{H}_{1}^{u_{1}}\left( X_{1},Y_{1}\right) +\mathrm{\mathrm{Ric}}%
^{1}\left( X_{1},Y_{1}\right) -\frac{n_{2}}{f}\mathrm{H}_{1}^{f}\left(
X_{1},Y_{1}\right) &=&\lambda g_{1}\left( X_{1},Y_{1}\right) \\
\mathrm{H}_{1}^{\phi _{1}}\left( X_{1},Y_{1}\right) +\mathrm{\mathrm{Ric}}%
^{1}\left( X_{1},Y_{1}\right) &=&\lambda g_{1}\left( X_{1},Y_{1}\right)
\end{eqnarray*}%
where $\phi _{1}=u_{1}-n_{2}\ln f$ and $u_{1}=u$ at a fixed point of $M_{2}$%
. Thus $\left( M_{1},g_{1},\phi _{1},\lambda \right) $ is a gradient Ricci
soliton. Now let $X=X_{2}$ and $Y=Y_{2}$, then%
\begin{eqnarray*}
\mathrm{H}^{u}\left( X_{2},Y_{2}\right) +\mathrm{\mathrm{Ric}}\left(
X_{2},Y_{2}\right) &=&\lambda g\left( X_{2},Y_{2}\right) \\
\mathrm{H}_{2}^{\phi _{2}}\left( X_{2},Y_{2}\right) +\mathrm{\mathrm{Ric}}%
^{2}\left( X_{2},Y_{2}\right) -f^{\sharp }g_{2}\left( X_{2},Y_{2}\right)
&=&\lambda f^{2}g_{2}\left( X_{2},Y_{2}\right) \\
\mathrm{H}_{2}^{\phi _{2}}\left( X_{2},Y_{2}\right) +\mathrm{\mathrm{Ric}}%
^{2}\left( X_{2},Y_{2}\right) &=&\left( \lambda f^{2}+f^{\sharp }\right)
g_{2}\left( X_{2},Y_{2}\right) \\
\mathrm{H}_{2}^{\phi _{2}}\left( X_{2},Y_{2}\right) +\mathrm{\mathrm{Ric}}%
^{2}\left( X_{2},Y_{2}\right) &=&\lambda _{2}g_{2}\left( X_{2},Y_{2}\right)
\end{eqnarray*}%
where $u_{2}=u$ at a fixed point of $M_{1}$ and $\lambda _{2}=\lambda
f^{2}+f^{\sharp }$ and $f^{\sharp }=f\Delta f+\left( n_{2}-1\right)
\left\Vert \nabla f\right\Vert _{1}^{2}$. If $f$ is constant, then $\left(
M_{1},g_{2},u_{2},\lambda f^{2}\right) $ is a gradient Ricci soliton.
\end{proof}

\section{Ricci Solitons on Warped Space-times}

In this section we will consider Ricci soliton on two well-known space-time
models, namely generalized Robertson-Walker space-times and standard static
space-times. More explicitly, we state some necessary conditions for these
models to be Ricci soliton. Then we also explore implications of that on the
components of the underlying models.

We begin by the following straightforward result.

\begin{remark}
Let $I$ be an open and connected subinterval of $\mathbb{R}$ furnished with $%
-\mathrm{d}t^2.$ Suppose that $u \partial_t \in \mathfrak{X}(I)$ is a vector
field on $I$ where $u \colon I \to \mathbb{R}$ is smooth. Then

\begin{enumerate}
\item $u \partial_t$ is concurrent on $(I, -\mathrm{d}t^2)$ if and only if $%
u(t)=t+a$ for some $a.$

\item $u\partial _{t}$ is conformal on $(I,-\mathrm{d}t^{2})$ with the
conformal factor $\mu =2u^{\prime }.$
\end{enumerate}
\end{remark}

\subsection{Ricci Solitons on Generalized Robertson-Walker Space-times}

We first define generalized Robertson-Walker space-times. Let $(M,g)$ be an $%
n-$dimensional Riemannian manifold and $f:I\rightarrow (0,\infty )$ be a
smooth function. Then $(n+1)-$dimensional product manifold $I\times M$
furnished with the metric tensor
\begin{equation*}
\bar{g}=-\mathrm{d}t^{2}\oplus f^{2}g
\end{equation*}%
is called a generalized Robertson-Walker space-time and is denoted by $\bar{M%
}=I\times _{f}M$ where $I$ is an open, connected subinterval of $\mathbb{R}$
and $\mathrm{d}t^{2}$ is the Euclidean metric tensor on $I$. This structure
was introduced to the literature to extend Robertson-Walker space-times \cite%
{Sanchez98, Sanchez99}. From now on, we will denote $\frac{\partial }{%
\partial t}\in \mathfrak{X}(I)$ by $\partial _{t}$ to state our results in
simpler forms.

Let $\bar{\zeta}=u\partial _{t}+\zeta $ be a concurrent vector field on a
generalized Robertson-Walker space-time $\overline{M}=I\times _{f}M$
furnished with the metric $\overline{g}=-\mathrm{d}t^{2}\oplus f^{2}g$, then%
\begin{equation*}
\bar{D}_{\bar{X}}\bar{\zeta}=\bar{X}
\end{equation*}%
for any $\bar{X}=x\partial _{t}+X\in \mathfrak{X}\left( \bar{M}\right) $.
This equation implies that%
\begin{eqnarray*}
\bar{D}_{\partial _{t}}\bar{\zeta} &=&\partial _{t} \\
\bar{D}_{\partial _{i}}\bar{\zeta} &=&\partial _{i}
\end{eqnarray*}%
where $\{\partial _{i} | \, i=1,2,...,n \}$ is an orthonormal set of vector
fields on $M$. The first equation yields%
\begin{equation}
\dot{u}\partial _{t}+\partial _{t}\left( \ln f\right) \zeta =\partial _{t}
\label{E41}
\end{equation}%
and the second equation yields%
\begin{equation}
u\partial _{t}\left( \ln f\right) \partial _{i}+D_{\partial _{i}}\zeta -f%
\dot{f}\zeta ^{i}\partial _{t}=\partial _{i}  \label{E42}
\end{equation}%
where $\zeta ^{i}=g\left( \zeta ,\partial _{i}\right) $. Equation (\ref{E41}%
) implies that $\dot{u}=1$ and $\dot{f}\zeta =0$. Therefore, $u=t+a$ where $%
a\in
\mathbb{R}
$. Now, we have two cases, namely, $\dot{f}=0$ or $\zeta =0$.

The first case implies that $D_{X}\zeta =X.$ That is, $\zeta $ is concurrent
on $M$. The second case implies that $u\dot{f}=f$ and so $f(t)=b\left(
t+a\right) $ where $b>0$ and $t+a>0$.

\begin{theorem}
Let $\bar{\zeta}=u\partial _{t}+\zeta $ be a field on a generalized
Robertson-Walker space-time $\overline{M}=I\times _{f}M$ furnished with the
metric $\overline{g}=-\mathrm{d}t^{2}\oplus f^{2}g$. Then $\bar{\zeta}$ is a
concurrent vector field on $\bar{M}$ if and only if $u=t+a$ and one of the
following conditions hold

\begin{enumerate}
\item $\zeta $ is concurrent on $M$ and $\dot{f}=0$.

\item $\zeta =0$ and $f=b\left( t+a\right) $ where $b>0$ and $t+a>0$.
\end{enumerate}
\end{theorem}

\begin{theorem}
Let $\overline{M}=I\times _{f}M$ be a generalized Robertson-Walker
space-time equipped with the metric $\overline{g}=-\mathrm{d}t^{2}\oplus
f^{2}g.$ If $\left( \bar{M}, \bar{g},u,\lambda \right) $ is a Ricci soliton
where%
\begin{equation*}
u=\int_{a}^{t}f\left( r\right) \mathrm{d}r, \quad \text{for some constant}
\quad a \in I
\end{equation*}%
then%
\begin{equation}
\mathrm{\mathrm{Ric}}=\left( \lambda -\dot{f}\right) g
\end{equation}
\end{theorem}

\begin{proof}
Let $\zeta =\mathrm{grad}u,$ then $\zeta =f\left( t\right) \partial _{t}$.
It is clear that the vector field is perpendicular to $M$. Suppose that $\{
\partial _{t},\partial _{1},\partial _{2},...,\partial _{m}\}$ is an
orthogonal basis for $\mathfrak{X}\left( \overline{M}\right) ,$ then the
Hessian tensor of $u$ is given by%
\begin{equation*}
\mathrm{H}^{u}\left( \partial _{t},\partial _{t}\right) =\bar{g}\left( D_{X}%
\mathrm{grad}u,Y\right)
\end{equation*}%
Now we have the following cases. The first case when $X=Y=\partial _{t}$. In
this case we have%
\begin{eqnarray*}
H^{u}\left( \partial _{t},\partial _{t}\right) &=&\bar{g}\left( D_{\partial
_{t}}\mathrm{grad}u,\partial _{t}\right) \\
&=&\dot{f}\bar{g}\left( \partial _{t},\partial _{t}\right)
\end{eqnarray*}%
The second case when $X=\partial _{t}$ and $Y=\partial _{i}$ for any $%
i=1,2,...,m$. In this case%
\begin{eqnarray*}
H^{u}\left( \partial _{t},\partial _{i}\right) &=&\bar{g}\left( D_{\partial
_{t}}\mathrm{grad}u,\partial _{i}\right) \\
&=&\dot{f}\bar{g}\left( \partial _{t},\partial _{i}\right)
\end{eqnarray*}%
Finally, $X=\partial _{i}$ and $Y=\partial _{j}$ for any $i,j=1,2,...,m$. In
this case%
\begin{eqnarray*}
H^{u}\left( \partial _{i},\partial _{j}\right) &=&\bar{g}\left( D_{\partial
_{i}}\mathrm{grad}u,\partial _{j}\right) \\
&=&f\bar{g}\left( D_{\partial _{i}}\partial _{t},\partial _{j}\right) \\
&=&f\bar{g}\left( \frac{\dot{f}}{f}\partial _{i},\partial _{j}\right) =\dot{f%
}\bar{g}\left( \partial _{i},\partial _{j}\right)
\end{eqnarray*}%
Thus $\mathrm{H}^{u}\left( X,Y\right) =\dot{f}\bar{g}\left( X,Y\right) $ and
hence%
\begin{eqnarray*}
\left( \mathcal{L}_{\zeta }\bar{g}\right) \left( X,Y\right) &=&\bar{g}\left(
D_{X}\mathrm{grad}u,Y\right) +\bar{g}\left( D_{Y}\mathrm{grad}u,X\right) \\
&=&2\mathrm{H}^{u}\left( X,Y\right) =2\dot{f}\bar{g}\left( X,Y\right)
\end{eqnarray*}%
Suppose that $\left( \bar{M},\bar{g},u,\lambda \right) $ is a Ricci soliton,
then%
\begin{eqnarray*}
\frac{1}{2}\mathfrak{\mathcal{L}}_{\zeta }\bar{g}+\mathrm{\mathrm{Ric}}
&=&\lambda \bar{g} \\
\dot{f}\bar{g}+\mathrm{\mathrm{Ric}} &=&\lambda \bar{g} \\
\mathrm{\mathrm{Ric}} &=&\left( \lambda -\dot{f}\right) \bar{g}
\end{eqnarray*}
\end{proof}

\begin{corollary}
Let $\overline{M}=I\times _{f}M$ be a generalized Robertson-Walker
space-time equipped with the metric $\overline{g}=-\mathrm{d}t^{2}\oplus
f^{2}g.$ Suppose that $\left( \bar{M},\bar{g},u,\lambda \right) $ is a
gradient Ricci soliton where%
\begin{equation*}
u=\int_{a}^{t}f\left( r\right) \mathrm{d}r, \quad \text{for some constant}
\quad a \in I.
\end{equation*}
Then

\begin{enumerate}
\item $\left( \bar{M},\bar{g}\right) $ is Einstein if $\dot{f}$ is constant

\item $\left( \bar{M},\bar{g}\right) $ is Ricci flat if $\lambda =\dot{f}$.
\end{enumerate}
\end{corollary}

\begin{theorem}
Let $\overline{M}=I\times _{f}M$ be a generalized Robertson-Walker
space-time equipped with the metric $\overline{g}=-dt^{2}\oplus f^{2}g.$ If $%
\left( \bar{M},\bar{g},\zeta ,\lambda \right) $ is a Ricci soliton with
concurrent vector field $\zeta $, then $\left( M,g\right) $ is Einstein with
factor $\left( n-1\right) c^{2}$ where $c=\left\Vert \mathrm{grad}%
f\right\Vert_1 $ is a constant.
\end{theorem}

\subsection{Ricci Solitons on Standard Static Space-times}

We begin by defining standard static space-times. Let $(M,g)$ be an $n-$%
dimensional Riemannian manifold and $f:M\rightarrow (0,\infty )$ be a smooth
function. Then $(n+1)-$dimensional product manifold $I\times M$ furnished
with the metric tensor%
\begin{equation*}
\bar{g}=-f^{2}\mathrm{d}t^{2}\oplus g
\end{equation*}%
is called a standard static space-time and is denoted by $\bar{M}=
_{f}I\times M$ where $I$ is an open, connected subinterval of $\mathbb{R}$
and $dt^{2}$ is the Euclidean metric tensor on $I$.

Note that standard static space-times can be considered as a generalization
of the Einstein static universe\cite{AD1,AD,GES,Besse2008}. The following
propositions are well-known and so the proofs are omitted.

By using \textit{Theorem \ref{Thm-concurrent}}, one can obtain the following
result for standard static space-times.

\begin{theorem}
Let $\bar{\zeta}=u\partial _{t}+\zeta $ be a vector field on a standard
static space-time of the form $\overline{M}=_{f} I\times M$ furnished with
the metric $\overline{g}=-f^2 \mathrm{d}t^{2}\oplus g$. Then $\bar{\zeta}$
is a concurrent vector field on $\bar{M}$ if and only if $\zeta \in
\mathfrak{(}M)$ is concurrent on $M$ and one of the following conditions hold

\begin{enumerate}
\item $u=t+a$ (i.e, $u \partial t$ is concurrent on $I$) and $f$ is constant,

\item $u=0$ and $\zeta(f)=f$
\end{enumerate}
\end{theorem}

The next result can be considered as a consequence of \textit{Theorem \ref%
{thm-KVF}}.

\begin{theorem}
Let $\bar{M}= _{f} I\times M$ be a standard static space-time with the
metric tensor $\bar{g}=-f^{2}\mathrm{d}t^{2}\oplus g.$ Suppose that $\left(
\bar{M}, \bar{g}, \bar{\zeta},\lambda \right) $ is a Ricci soliton where $%
\bar{\zeta}=u\partial_{t}+\zeta.$ Then $(\bar M, \bar g)$ is Einstein if

\begin{enumerate}
\item $\bar \zeta=\zeta$ is Killing on $M$

\item $\bar \zeta = u \partial_t$ and $u=t+a$ (i.e, $u \partial_t$ is
Killing on $I$),

\item $u \partial_t$ and $\zeta$ are Killing vector fields on $I$ and $M,$
respectively and $\zeta(f)=0.$
\end{enumerate}
\end{theorem}

Now an application of \textit{Theorem \ref{thm-Riccisol}} yields that:

\begin{theorem}
Let $\bar{M}= _{f}I \times M$ be a standard static space-time with the
metric tensor $\bar{g}=-f^{2}\mathrm{d}t^{2}\oplus g.$ Suppose that $%
\left(M, g, \zeta,\lambda \right) $ is a Ricci soliton. Then $\left( \bar{M}%
, \bar{g}, \bar{\zeta},\lambda \right) $ is a Ricci soliton where $\bar{\zeta%
}=u\partial_{t}+\zeta$ if

\begin{enumerate}
\item $\mathrm{H}^{f}=0$ and,

\item $\zeta (f)=(\lambda -u^{\prime })f$
\end{enumerate}
\end{theorem}

The preceding corollary due to \textit{Theorem \ref{thm-imp}} is about the
implications of a gradient Ricci soliton standard static space-time.

\begin{corollary}
Let $\bar{M}=_{f}I\times M$ be a standard static space-time furnished with
the metric $\overline{g}=-f^{2}\mathrm{d}t^{2}\oplus g.$ Suppose that $%
\left( \bar{M},\bar{g},u,\lambda \right) $ is a gradient Ricci soliton. Then

\begin{enumerate}
\item $(M,g,\phi _{2},\lambda )$ is a gradient Ricci soliton with $\phi
_{2}=u_{1}-\ln f$ where $u_{1}=u$ at some fixed point of $I,$

\item $(I,-\mathrm{d}t^{2},\phi _{1},\lambda _{1})$ is a gradient Ricci
soliton with $\lambda _{1}=\lambda f^{2}$ and $\phi _{1}=u=t^{2}/2+at+b$ for
some $a$ and $b$ and also $\phi _{1}=u$ at some fixed point of $M$ when $f$
is constant.
\end{enumerate}
\end{corollary}

Now, we will obtain some results by computing the following fundamental
Ricci soliton equation on a standard static space-time of the form $%
\overline{M}=_{f}I\times M$ with the metric $\overline{g}=-f^{2}\mathrm{d}%
t^{2}\oplus g.$ Suppose that $\bar{\zeta}=u\partial _{t}+\zeta $ is a vector
field on $\bar{M}$ where $\zeta $ is a vector field on $M$ and $u\colon
I\rightarrow \mathbb{R}$ is smooth.

\begin{equation}
\frac{1}{2}\mathfrak{\mathcal{L}}_{\bar{\zeta}}\bar{g}+\mathrm{\mathrm{Ric}}%
=\lambda \bar{g}  \label{FE-SSST}
\end{equation}

Evaluating both sides of Equation \ref{FE-SSST} at $(\partial _{t},\partial
_{t}),$ we have that
\begin{equation*}
\Delta _{M}(f)=\left( u^{\prime }-\lambda \right) f+\zeta (f)
\end{equation*}%
It is clear that $\lambda =u^{\prime }$ if $f$ is constant. Equation\textit{%
\ \ref{FE-SSST}} at $(X,Y)$ where $X,Y\in \mathfrak{X}(M),$ yields%
\begin{equation*}
\frac{1}{2}\mathfrak{\mathcal{L}}_{\bar{\zeta}}^{M}g(X,Y)+\mathrm{\mathrm{Ric%
}}^{M}(X,Y)=\lambda g(X,Y)+\frac{1}{f}\mathrm{H}^{f}(X,Y)
\end{equation*}%
Moreover if $\zeta $ is conformal on $M$ with factor $2\rho $, then%
\begin{equation*}
\mathrm{\mathrm{Ric}}^{M}(X,Y)=\left( \lambda -\rho \right) g(X,Y)+\frac{1}{f%
}\mathrm{H}^{f}(X,Y)
\end{equation*}%
Thus $M$ is Einstein if $\mathrm{H}^{f}=0$. By taking the trace of both
sides one gets that%
\begin{equation*}
S=n\left( \lambda -\rho \right) +\frac{1}{f}\Delta _{M}(f)
\end{equation*}

\begin{theorem}
Let $\bar{M}=_{f}I\times M$ be a standard static space-time furnished with
the metric $\overline{g}=-f^{2}\mathrm{d}t^{2}\oplus g.$ Suppose that $%
\left( \bar{M},\bar{g},\bar{\zeta},\lambda \right) $ is a Ricci soliton. Then%
\begin{equation*}
\Delta _{M}(f)=\left( u^{\prime }-\lambda \right) f+\zeta (f)
\end{equation*}%
Moreover, if $\zeta $ is a conformal vector field on $M$ with factor $2\rho $%
, then the scalar curvature $S$ of $M$ is given by
\begin{equation*}
S=n\left( \lambda -\rho \right) +\frac{1}{f}\Delta _{M}(f)
\end{equation*}
\end{theorem}

\begin{corollary}
Let $\bar{M}=_{f}I\times M$ be a standard static space-time furnished with
the metric $\overline{g}=-f^{2}\mathrm{d}t^{2}\oplus g.$ Suppose that $%
\left( \bar{M},\bar{g},\bar{\zeta},\lambda \right) $ is a Ricci soliton.
Then $M$ is Einstein if $\zeta $ is conformal on $M$ and \thinspace\ $%
\mathrm{H}^{f}=0$.
\end{corollary}


\end{document}